\documentclass[preprint,14pt]{elsarticle}
\usepackage{amsfonts}
\usepackage{amsmath}

\newdefinition{definition}{Definition}
\newdefinition{remark}{Remark}
\newtheorem{theorem}{Theorem}
\newtheorem{lemma}{Lemma}
\newtheorem{corollary}{Corollary}
\newtheorem{proposition}{Proposition}
\newproof{proof}{Proof}

\usepackage{amssymb}

\journal{Advances in Mathematics}

\begin{document}

\begin{frontmatter}



\title{A new family of semifields with 2 parameters}


\author[ovgu,nudt]{Yue Zhou}
\ead{yue.zhou@st.ovgu.de}

\author[ovgu]{Alexander Pott\corref{cor1}}
\ead{alexander.pott@ovgu.de}
\cortext[cor1]{Corresponding author}

\address[ovgu]{Faculty of Mathematics, Otto-von-Guericke-University
Magdeburg, \\39106 Magdeburg, Germany}

\address[nudt]{Department of Mathematics and System Sciences, Science College,
National University of Defense Technology,
Changsha, 410073, P.R.China}

\begin{abstract}
    A new family of commutative semifields with two parameters is presented. Its left and middle nucleus are both determined. Furthermore, we prove that for any different pairs of parameters, these semifields are not isotopic. It is also shown that, for some special parameters, one semifield in this family can lead to two inequivalent planar functions. Finally, using similar construction, new APN functions are given.
\end{abstract}


\end{frontmatter}



\section{Introduction}
A \emph{semifield} $\mathbb{S}$ is an algebraic structure satisfying all the axioms of a skewfield except (possibly) associativity. In other words, it satisfies the following axioms:
    \begin{itemize}
      \item $(\mathbb{S},+)$ is a group, with identity element $0$;
      \item $(\mathbb{S}\setminus\{0\},*)$ is a quasigroup;
      \item $0*a=a*0=0$ for all $a$;
      \item The left and right distributive laws hold, namely for any $a,b,c\in\mathbb{S}$,
      $$(a+b)*c=a*c+b*c,$$
      $$a*(b+c)=a*b+a*c;$$
      \item There is an element $e\in \mathbb{S}$ such that $e*x=x*e=x$ for all $x\in \mathbb{S}$.
    \end{itemize}
A finite field is a trivial example of a semifield. Furthermore, if $\mathbb{S}$ does not necessarily have a multiplicative identity, then it is called a \emph{presemifield}. A semifield is not necessarily commutative or associative. However, by Wedderburn's Theorem \cite{Wedderburn1905}, in the finite case, associativity implies commutativity. Therefore, a non-associative finite commutative semifield is the closest algebraic structure to a finite field.

In the earlier literature, semifields were also called \emph{division rings} or \emph{distributive quasifields}. The study of semifields was initiated by Dickson, see \cite{Dickson1906}, shortly after the classification of finite fields. Until now, semifields have become an attracting topic in many different areas of mathematics, such as difference sets, coding theory and finite geometry.

Dickson constructed the first non-trivial semifields in \cite{Dickson1906}. In \cite{Knuth1963}, Knuth showed that the additive group of a semifield $\mathbb{S}$ is an elementary abelian group, and the additive order of the nonzero elements in $\mathbb{S}$ is called the characteristic of $\mathbb{S}$. Hence, any finite semifield can be represented by $(\mathbb{F}_{p^n}, +, *)$. Here $(\mathbb{F}_{p^n}, +)$ is the additive group of the finite field $\mathbb{F}_{p^n}$ and $x*y=\varphi(x,y)$, where $\varphi$ is a mapping from $\mathbb{F}_{p^n}\times \mathbb{F}_{p^n}$ to $\mathbb{F}_{p^n}$.

On the other hand, there is a well-known correspondence, via coordinatisation, between semifields and projective planes of Lenz-Barlotti type V.1, see \cite{HughesPiper}. In \cite{Albert1960}, Albert showed that two semifields coordinatise isomorphic planes if and only if they are isotopic (isotopism can also be defined between presemifields):
\begin{definition}
    Let $\mathbb{S}_1=(\mathbb{F}_{p^n}, +, *)$ and $\mathbb{S}_2=(\mathbb{F}_{p^n}, +, \star)$ be two presemifields. If there exist three bijective linear mapping $L, M, N:\mathbb{F}_{p}^n\rightarrow \mathbb{F}_{p}^n$ such that
    $$M(x)\star N(y)=L(x*y)$$
    for any $x,y\in\mathbb{F}_{p^n}$, then $\mathbb{S}_1$ and $\mathbb{S}_2$ are called \emph{isotopic}, and the triple $(M,N,L)$ is an \emph{isotopism } between $\mathbb{S}_1$ and $\mathbb{S}_2$. Furthermore, if there exists an isotopism of the form $(N, N, L)$ between $\mathbb{S}_1$ and $\mathbb{S}_2$, then $\mathbb{S}_1$ and $\mathbb{S}_2$ are \emph{strongly isotopic}.
\end{definition}

Let $\mathbb{P}=(\mathbb{F}_{p^n}, +, *)$ be a presemifield, and $a\in \mathbb{P}$. If we define a new multiplication $\star$ by the rule
    $$(x*a)\star(a*y)=x*y,$$
we have $(a*a)\star(a*x)=a*x$ and $(x*a)\star(a*a)=x*a$, namely $(\mathbb{F}_{p^n}, +, \star)$ is a semifield with unit $a*a$. There are many semifields associated with a presemifield, but they are all isotopic.

Let $\mathbb{S}=(\mathbb{F}_{p^n},+,*)$ be a semifield. The subsets
\begin{equation*}
\begin{aligned}
  N_l(\mathbb{S})=\{a\in \mathbb{S}: (a*x)*y=a*(x*y) \text{ for all }x,y\in \mathbb{S}\},\\
  N_m(\mathbb{S})=\{a\in \mathbb{S}: (x*a)*y=x*(a*y) \text{ for all }x,y\in \mathbb{S}\},\\
  N_r(\mathbb{S})=\{a\in \mathbb{S}: (x*y)*a=x*(y*a) \text{ for all }x,y\in \mathbb{S}\},
\end{aligned}
\end{equation*}
are called the \emph{left, middle} and \emph{right nucleus} of $\mathbb{S}$, respectively. It is easy to check that these sets are finite fields. The subset $N(\mathbb{S})=N_l(\mathbb{S})\cap N_m(\mathbb{S}) \cap N_r(\mathbb{S})$ is called the \emph{nucleus} of $\mathbb{S}$ (In some papers, not $N(\mathbb{S})$ but $N_m(\mathbb{S})$ is called the \emph{nucleus} of $\mathbb{S}$). It is easy to see, if $\mathbb{S}$ is commutative, then $N_l(\mathbb{S})=N_r(\mathbb{S})$ and $N_l(\mathbb{S})\subseteq N_m(\mathbb{S})$, therefore $N_l(\mathbb{S})=N_r(\mathbb{S})=N(\mathbb{S})$ . In \cite{HughesPiper}, geometric interpretation of these nuclei is discussed.

Next, we give the definition of planar functions, which was introduced by Dembowski and Ostrom in \cite{DO1968} to describe affine planes possessing a collineation group with specific properties.
\begin{definition}
    Let $p$ be an odd prime. A function $f: \mathbb{F}_{p^n}\rightarrow\mathbb{F}_{p^n}$ is called a \emph{planar function}, or \emph{perfect nonlinear (PN)}, if for each $a\in \mathbb{F}_{p^n}^*$, $f(x+a)-f(x)$ is a bijection on $\mathbb{F}_{p^n}$.
\end{definition}

For $p=2$, if $x_0$ is a solution of $f(x+a)-f(x)=b$, then $x_0+a$ is another one, hence there do not exist planar functions over $\mathbb{F}_{2^n}$. A function $f: \mathbb{F}_{2^n}\rightarrow\mathbb{F}_{2^n}$ is called \emph{almost perfect nonlinear (APN)}, if for each pair of $a,b\in \mathbb{F}_{p^n}$ where $a\neq 0$,  $\#\{x\mid f(x+a)-f(x)=b\}=2$. APN functions have important applications in cryptography, for recent surveys, see \cite{Carlet_Vect_Book} and \cite{EdelPott2009}.

A \emph{Dembowski-Ostrom} (DO) polynomial $D\in \mathbb{F}_{p^n}[x] $ is a polynomial
$$D(x)=\sum_{i,j}a_{ij}x^{p^i+p^j}\enspace .$$
Obviously, $D(0)=0$ and $D(x+a)-D(x)-D(a)$ is a linearized polynomial for any nonzero $a$ (see \cite{LidlNiederreiter} for linearized polynomials over finite fields). It can be proved that a planar DO polynomial is equivalent to a commutative presemifield with odd characteristic, see \cite{Coulter2008}. In fact, if $*$ is the presemifield product, then the corresponding planar function is $f(x)=x*x$; when the planar DO polynomial $f$ is given, then the corresponding presemifield product can be defined as
\begin{equation}\label{eq_PN->production}
    x*y=\frac{1}{2}(f(x+y)-f(x)-f(y))\enspace .
\end{equation}
It is remarkable that all the known planar functions are DO polynomials, except for the one found by Coulter and Matthews in \cite{Coulter-Matthews1997}, which defines some planes of Lenz-Barlotti class II, but not a semifield plane.

A function from a finite field $\mathbb{F}_{p^n}$ to itself is \emph{affine}, if it is defined by the sum of a constant and a linearized polynomial over $\mathbb{F}_{p^n}$. There are several equivalence relations of functions for which the \emph{planar} property is invariant:
\begin{definition}
    Two functions $f$ and $g: \mathbb{F}_{p^n}\rightarrow \mathbb{F}_{p^n}$ are called
    \begin{itemize}
      \item \emph{extended affine equivalent} (EA-equivalent), if $g = l_1 \circ f \circ l_2+l_3$, where $l_1, l_2$ and $l_3$ are affine functions, and where $l_1,l_2$ are permutations of $\mathbb{F}_{p^n}$. Furthermore, if $l_3$ is the zero mapping, then $f$ and $g$ are called \emph{affine equivalent}; if $l_1$ and $l_2$ are both linear, and $l_3$ is the zero mapping, then $f$ and $g$ are called \emph{linear equivalent};
      \item \emph{Carlet-Charpin-Zinoviev equivalent} (CCZ-equivalent or graph equivalent), if there is some affine permutation $L$ of $\mathbb{F}_{p}^{2n}$, such that $L(G_f)=G_g$, where $G_f=\{(x,f(x)):x\in\mathbb{F}_{p^n})\}$ and $G_g=\{(x,g(x)):x\in\mathbb{F}_{p^n})\}$.
    \end{itemize}
\end{definition}

Generally speaking, EA-equivalence implies CCZ-equivalence, but not vice versa, see \cite{Budaghyan06}. However, if planar functions $f$ and $g$ are CCZ-equivalent, then they are also EA-equivalent \cite{B-H2008,EA=CCZ_PN2008}. Moreover, it is easy to prove that:
\begin{lemma}
    Let $f$ and $g$ be both planar DO functions from $\mathbb{F}_{p^n}$ to $\mathbb{F}_{p^n}$. Then $f$ and $g$ are EA-equivalent if and only if $f$ and $g$ are linear equivalent.
\end{lemma}
\begin{proof}
    Since linear equivalence is a special case of  EA-equivalence, we only need to prove the contrary statement.
    Now assume that $f$ and $g$ are EA-equivalent, i.e. there is affine functions $l_1$, $l_2$ and $l_3$ such that \begin{equation}\label{eq_EA=linear_DO}
      g = l_1 \circ f \circ l_2+l_3,
    \end{equation}
    where $l_1$ and $l_2$ are both permutations. Notice that if $f$ is a planar DO function, there exists a presemifield multiplication $*$ such that, $f(x)=x*x$. Let $l_i(x)=\bar{l}_i(x)+a_i$, where $\bar{l}_i(0)=0$ for $i=1,2$. Then the right side of (\ref{eq_EA=linear_DO}) becomes:
    \begin{align*}
      &l_1\circ f(\bar{l}_2(x)+a_2)+l_3\\
      =&l_1\left(\bar{l}_2(x)*\bar{l}_2(x)+2\bar{l}_2(x)*a_2+a_2*a_2\right)+l_3\\
      =&\bar{l}_1\left(\bar{l}_2(x)*\bar{l}_2(x)\right)+2\bar{l}_1\left(\bar{l}_2(x)*a_2\right)+\bar{l}_1(a_2*a_2)+a_1+l_3.
    \end{align*}
    According to the distributivity of presemifield, $\bar{l}_1(\bar{l}_2(x)*\bar{l}_2(x))$ is also a DO function, namely a quadratic form, and the rest part of the equation above is affine. However, as the right side of (\ref{eq_EA=linear_DO}) is a planar DO function, we have
    $$g(x)=\bar{l}_1\left(\bar{l}_2(x)*\bar{l}_2(x)\right)=\bar{l}_1\circ f \circ \bar{l}_2(x),$$
    which means that $f$ and $g$ are linear equivalent.\qed
\end{proof}

Furthermore, because of the correspondence between commutative presemifields with odd characteristic and planar functions, as we mentioned above, the strong isotopism of two commutative presemifields is equivalent to the linear equivalence of the corresponding  planar DO functions, which we call directly the \emph{equivalence} of planar DO functions.

In the end of this section, we list all the commutative semifields of order ${p^n}$ that are known. For any odd $p$, there are:
\begin{enumerate}
  \item Finite Field;
  \item Albert commutative twisted field \cite{Albert52};
  \item Dickson semifield \cite{Dickson1906};
  \item Budaghyan-Helleseth semifield \cite{B-H2008}, with $n$ even (also discovered independently by Zha and Wang in \cite{Zha2009-1});
  \item Zha-Kyureghyan-Wang semifield \cite{Zha2009}, with $n=3k$;
  \item Bierbrauer semifield \cite{BierBrauer2009-2}, with $n=4k$;
  \item Lunardon-Marino-Polverino-Tormbetti semifield \cite{LMPT2009}, with $n$ even (generalized by Bierbrauer in \cite{Bierbrauer2009-3}).
\end{enumerate}

\begin{remark}
In \cite{B-H2008}, Budaghyan and Helleseth present two families of planar functions, but in \cite{Bierbrauer2009-3}, Bierbrauer proves that one of them belongs to the other. Therefore, we consider them as one family.
\end{remark}
\begin{remark}
    It should be mentioned that, the non-isotopism between (4) and (7) seems to be open yet. The first author showed in \cite{Zhou2010} that for $n=6$ and $p=3$, these two semifields are isotopic. However, it is beyond our computer capacity to check the isotopisms for $n>6$.
\end{remark}

For $p=3$, there are:
\begin{enumerate}
    \setcounter{enumi}{7}
    \item Coulter-Matthews and Ding-Yuan semifield \cite{Coulter-Matthews1997,Ding-Yuan2006}, with $n$ odd;
    \item Cohen-Ganley semifield \cite{Cohen-Ganley1982}, with $n$ odd;
  \item Ganley semifield \cite{Ganley1981}, with $n$ odd.
\end{enumerate}

Sporadic examples (for (13), (14) and (15), we only give the corresponding planar functions which are discovered by Weng and Zeng):
\begin{enumerate}
  \setcounter{enumi}{10}
  \item Coulter-Henderson-Kosick semifield \cite{Coulter-Henderson-Kosick2007}, with $p=3$ and $n=8$ ;
  \item Penttila-Williams semifield \cite{Penttila-Williams2004}, with $p=3$ and $n=10$;
  \item $x^{90}+x^2$ on $\mathbb{F}_{3^5}$ \cite{Weng2010};
  \item $x^{162}+x^{108}-x^{84}+x^2$ on $\mathbb{F}_{3^5}$ \cite{Weng2010};
  \item $x^{50}+3x^6$ on $\mathbb{F}_{5^5}$ \cite{Weng2010}.
\end{enumerate}

\section{Semifield Family with two parameters}
In \cite{Cohen-Ganley1982}, Cohen and Ganley made significant progress in the investigation of commutative semifields of rank 2 over their middle nucleus. Here the ``rank 2'' means that if the size of semifield is $p^{2m}$, then its middle nucleus is of size $p^m$. Let $a,b,c,d\in\mathbb{F}_{p^m}$, $n=2m$. Cohen and Ganley defined a binary mapping $*$ from $ \mathbb{F}_{p^n}\times \mathbb{F}_{p^n}$ to $ \mathbb{F}_{p^n}$ as follows:
\begin{equation}\label{eq_CohenGanley_multi}
    (a,b)*(c,d)=(ac+\varphi_1(bd), ad+bc+\varphi_2(bd)),
\end{equation}
where $\varphi_1$ and $\varphi_2$ are linearized polynomials.
Furthermore, they considered under which condition $*$ defines the multiplication of a semifield.
Some sufficient and necessary conditions were derived, see \cite{BallLavrauw2002,Cohen-Ganley1982} for detail. Furthermore, finite fields, Dickson semifields, Cohen-Ganley semifields and the Penttila-Williams semifield are all of this form.

We see that the multiplication of $\mathbb{F}_{p^m}$ is used in the multiplication $*$ defined by (\ref{eq_CohenGanley_multi}), which is basically a linear combination of $ac$, $bd$, $ad$ and $bc$. Hence, one natural question arises: Is it possible to construct some semifields or presemifields, if we replace some of these finite field multiplications by semifield or presemifield multiplications? Our first candidate is naturally the multiplication of Albert twisted fields, and it turns out to work quite well.
\begin{theorem}\label{th_Mola_presemifield}
    Let $p$ be an odd prime, and let $m,k$ be positive integers, such that $\frac{m}{\gcd(m,k)}$ is odd. Define $x\circ_k y=x^{p^k}y+y^{p^k}x$. For elements $(a,b),(c,d)\in\mathbb{F}_{p^{m}}^2$, define a binary operation $*$ as follows:
    \begin{equation}\label{eq_mola_multi}
        (a,b)*(c,d)=(a\circ_k c+\alpha (b\circ_k d)^{\sigma}, ad+bc),
    \end{equation}
    where $\alpha$ is a non-square element in $\mathbb{F}_{p^m}$ and $\sigma$ is a field automorphism of $\mathbb{F}_{p^m}$. Then, $(\mathbb{F}_{p^{2m}},+,*)$ is a presemifield, which we denote by $\mathbb{P}_{k,\sigma}$.
\end{theorem}
\begin{proof}
    It is routine to check the distributive law of $\mathbb{P}_{k,\sigma}$. Hence, to prove $\mathbb{P}_{k,\sigma}$
    is a presemifield, we only need to prove that
    $$(a,b)*(c,d)=0\hbox{ if and only if }a=b=0\hbox{ or }c=d=0.$$

    Assume that $(a,b)*(c,d)=0$, then we have
    \begin{eqnarray}\label{eq_mola_both0}
      a\circ_k c+\alpha (b\circ_k d)^{\sigma} &=& 0~, \\
    \nonumber  ad+bc &=& 0~.
    \end{eqnarray}

    When $d=0$, we have $a\circ_k c=0$ and $bc=0$, which means $c=0$ or $a=b=0$ since $\circ_k$ is the Albert presemifield multiplication on $\mathbb{F}_{p^m}$.

    When $d\neq 0$, we have $a =-\frac{bc}{d}$. If $b=0$, then $a=0$. If $b\neq0$, then eliminating $a$ in (\ref{eq_mola_both0}), we have
    $$\alpha(b^{p^k}d+d^{p^k}b)^\sigma=c^{p^k+1}\left(\frac{b}{d}+\left(\frac{b}{d}\right)^{p^k}\right),$$
    which means that
    $$\alpha=(c^{p^k+1}(d^{-\sigma})^{p^k+1})\left(\frac{b}{d}+\left(\frac{b}{d}\right)^{p^k}\right)^{1-\sigma}.$$
    However, the equation cannot hold, since $\alpha$ is a non-square in $\mathbb{F}_{p^m}$. Therefore, we get $a=b=0$. \qed
\end{proof}

To analysis the properties of $\mathbb{P}_{k,\sigma}$, we need the following well-known results:
\begin{proposition}\label{pro_gcd(p^m-1,p^k+1)}
    Let $p$ be an odd prime, and let $m,k$ be positive integers, such that $\frac{m}{\gcd(m,k)}$ is odd. Then
    \begin{enumerate}
        \item $\gcd(p^m-1,p^k+1)=2$, which means that $x^{p^k+1}$ is a 2-1 mapping  on $\mathbb{F}_{p^m}$;
        \item $x^{p^k}+x$ is a permutation on $\mathbb{F}_{p^m}$.
    \end{enumerate}
\end{proposition}

It is easy to see that different  $\alpha$ generate isotopic semifields, because
\begin{align*}
  &(a\circ_k c+\alpha (\beta b\circ_k \beta d)^{\sigma}, a(\beta d)+c(\beta b))\\
  =&(a\circ_k c+\alpha\beta^{(p^k+1)\sigma} (b\circ_k d)^{\sigma}, \beta(ad+bc)),
\end{align*}
and the image set of $(\cdot)^{p^k+1}$ are all the squares in $\mathbb{F}_{p^m}$.
In the remaining part, we assume that the non-square $\alpha$ is an element of  $\mathbb{F}_{p^k}\cap\mathbb{F}_{p^m}=\mathbb{F}_{p^l}$, $l=\gcd(k,m)$.
Furthermore,
the planar function that corresponds to $\mathbb{P}_{k,\sigma}$ is,
$$    (x,y)\mapsto
    (   2x^{p^k+1}+2\alpha (y^{p^k+1})^\sigma ,
        2xy ).$$
Divide it by 2, we have
\begin{equation*}
    (x,y)\mapsto
    (   x^{p^k+1}+\alpha (y^{p^k+1})^\sigma ,
        xy )
\end{equation*}
If there exists some $u$ such that $p^k+1=p^u(p^s+1)\mod p^m-1$, then $\mathbb{P}_{k,\sigma}$ is isotopic to $\mathbb{P}_{s,\sigma}$, since
$$    (x,y)\mapsto
    (   x^{p^k+1}+\alpha (y^{p^k+1})^\sigma ,
        xy ),$$
is equivalent with
$$    (x,y)\mapsto
    (   (x^{p^u})^{p^s+1}+\alpha ((y^{p^u})^{p^s+1})^\sigma ,
        (x^{p^u}y^{p^u})^{p^{-u}} ),$$
which is also equivalent with
$$    (x,y)\mapsto
    (   x^{p^s+1}+\alpha (y^{p^s+1})^\sigma ,
        xy ).$$
For different $\sigma$ but the same $k$, the same result can also be derived. Hence, in the rest of this paper, we always let  $0\le k, r\le \lfloor\frac{m}{2}\rfloor$ and $k\neq 0$, where $\sigma(x)=x^{p^r}$.

To get a semifield $\mathbb{S}_{k,\sigma}$ from our presemifield, we can define the multiplication $\star$ of $\mathbb{S}_{k,\sigma}$ as follows:
\begin{equation}\label{eq_mola_semifield_multi}
     ((a,b)*(1,0))\star((c,d)*(1,0)):=(a,b)*(c,d)
\end{equation}
Let $L(a,b)=(a,b)*(1,0)=(a+a^{p^k},b)$, which is a linear mapping and permutation on $\mathbb{F}_{p^{2m}}$ by Proposition \ref{pro_gcd(p^m-1,p^k+1)}.
For convenience, when $\sigma$ is the identity mapping on $\mathbb{F}_{p^m}$, we will denote our presemifield and semifield by $\mathbb{P}_{k}$ and $\mathbb{S}_{k}$.

If $L(\beta)=y\in\mathbb{F}_{p^{2m}}$ is in the middle nucleus of $\mathbb{S}_{k,\sigma}$, then for any $x,y \in \mathbb{F}_{p^{2m}}$ we have,
$$(x\star L(\beta))\star z=x\star(L(\beta)\star z),$$
since $L$ is a permutation. It is equivalent to
$$(L(x)\star L(\beta))\star L(z)=L(x)\star(L(\beta)\star L(z)),$$
which is also,
$$L^{-1}(x*\beta)*z=x*L^{-1}(\beta*z).$$
Furthermore, we can precisely determine the middle nucleus of $\mathbb{S}_{k,\sigma}$:
\begin{theorem}\label{th_S_Nm}
    Let $\mathbb{S}_{k,\sigma}$ be the semifield with the multiplication $\star$ defined on $\mathbb{F}_{p^{2m}}$ as in (\ref{eq_mola_semifield_multi}), with $\alpha\in\mathbb{F}_{p^l}$, $\alpha\neq0$ and $l=\gcd(m,k)$, then
    \begin{enumerate}
      \item If $\sigma$ is the identity mapping, then the middle nucleus $N_m(\mathbb{S}_k)$ is isomorphic to $\mathbb{F}_{p^{2l}}$;
      \item If $\sigma$ is not trivial, then the middle nucleus $N_m(\mathbb{S}_{k,\sigma})$ is isomorphic to $\mathbb{F}_{p^{l}}$.
    \end{enumerate}
\end{theorem}
\begin{proof}
    Assume that $L(c,d)\in N_m(\mathbb{S}_{k,\sigma})$, then we have
    \begin{equation}\label{eq_S_Nm_1}
        L^{-1}((a,b)*(c,d))*(e,f)= (a,b)*L^{-1}((c,d)*(e,f)),
    \end{equation}
    for any $(a,b),(e,f) \in \mathbb{F}_{p^{m}}^2$. Since $x^{p^k}+x$ is a permutation, we can define $u$ by
    \begin{equation}\label{eq_S_u}
        u+u^{p^k}=a\circ_k c+\alpha (b\circ_k d)^\sigma,
    \end{equation}
    and we get
    \begin{align}\label{eq_Mi_(xy)z}
        \nonumber&L^{-1}((a,b)*(c,d))*(e,f)\\
                 =&(u,ad+bc)*(e,f)\\
        \nonumber=&(u\circ_ke+\alpha(f\circ_k(ad+bc))^\sigma, uf+(ad+bc)e).
    \end{align}
    Similarly, we define $v$ by
    \begin{equation}\label{eq_S_v}
        v+v^{p^k}=c\circ_k e+\alpha (d\circ_k f)^\sigma,
    \end{equation}
    and the right side of (\ref{eq_S_Nm_1}) is
    \begin{align}\label{eq_Mi_x(yz)}
        \nonumber &(a,b)*L^{-1}((c,d)*(e,f))\\
        =&(a\circ_kv+\alpha(b\circ_k(cf+de))^\sigma, vb+a(cf+de)).
    \end{align}
    By comparing the second component of two sides of (\ref{eq_S_Nm_1}), we have
    $$uf+bce=vb+acf.$$
    For $f=0$ but $b\neq 0$, we have $ceb=vb$, which means that $v=ce$. Eliminate $v$ in (\ref{eq_S_v}), we have
    $$ce+(ce)^{p^k}=c^{p^k}e+ce^{p^k},$$
    for any $e\in \mathbb{F}_{p^m}$. That means $c=c^{p^k}$, namely,
    \begin{equation}\label{eq_Mi_c_in_F_q^l}
        c\in \mathbb{F}_{p^l}.
    \end{equation}

    Furthermore, for $f\neq 0$, we have
    $$u+\frac{bce}{f}=\frac{b}{f}\cdot{v}+ac,$$
    by eliminating $u$ using (\ref{eq_S_u}) and (\ref{eq_S_v}), we have
    \begin{align*}
        &(a^{p^k}+a)c+\alpha(b^{p^k}d+bd^{p^k})^\sigma+c\left(\frac{b}{f}\cdot e+\left(\frac{b}{f}\cdot e\right)^{p^k}\right)\\
        =& \frac{b}{f}v+\left(\frac{b}{f}\right)^{p^k}(c\circ_k e+\alpha(d\circ_k f)^\sigma-v)+ac+(ac)^{p^k}\\
        =& \frac{b}{f}v+\left(\frac{b}{f}\right)^{p^k}(c(e+e^{p^k})+\alpha(d\circ_k f)^\sigma-v)+c(a+a^{p^k}).\\
    \end{align*}
    By canceling the same terms on both sides, we have
    \begin{equation}\label{eq_Mi_long1}
        \alpha \left((b\circ_k d)^\sigma-(d\circ_k f)^\sigma\left(\frac{b}{f}\right)^{p^k}\right)+\left( \frac{b}{f}-\left(\frac{b}{f}\right)^{p^k}\right)(ce-v)=0.
    \end{equation}

    If $\sigma$ is the identity mapping, then $$\left(\frac{b}{f}-\left(\frac{b}{f}\right)^{p^k}\right)(\alpha fd^{p^k}+ce-v)=0.$$
    Since the equation above should hold for any $b$ and $f\neq 0$, we have
    $$v=\alpha f d^{p^k}+ce,$$
    which means that
    $$v+v^{p^k}=(e^{p^k}+e)c+\alpha(f^{p^k}d^{p^{2k}}+fd^{p^k}),$$
    since $\alpha\in \mathbb{F}_{p^l}=\mathbb{F}_{p^k}\cap\mathbb{F}_{p^m}$. Together with (\ref{eq_S_v}), we have
    $$df^{p^k}+fd^{p^k}=f^{p^k}d^{p^{2k}}+fd^{p^k},$$
    for any $f\neq 0$, which means that $d=d^{p^k}$.
    Therefore, if $(c,d)\in N_m(\mathbb{S}_{k})$, then $c,d\in \mathbb{F}_{p^k}\cap\mathbb{F}_{p^m}(=\mathbb{F}_{p^l})$. Since the middle nucleus of a finite semifield is isomorphic to a finite field, $N_m(\mathbb{S}_{k})$ is isomorphic to a subfield of $\mathbb{F}_{p^{2l}}$. Conversely, it is routine to check that $(c,d)\in N_m(\mathbb{S}_{k})$ , for $c,d\in\mathbb{F}_{p^l}$. Therefore we proved the claim (1).

    If $\sigma$ is not trivial, then it is also routine to check that $(c,0)\in N_m(\mathbb{S}_{k,\sigma})$, for $c\in\mathbb{F}_{p^l}$, namely $\mathbb{F}_{p^{l}}$ is a subfield of $N_m(\mathbb{S}_{k,\sigma})$. Next, we are going to prove $d=0$. We separate this proof into two steps, first let us  prove that $d\in\mathbb{F}_{p^l}$.
    Let $L(c,d)\in N_m(\mathbb{S}_{k,\sigma})$, which means that $c\in\mathbb{F}_{p^l}$, see (\ref{eq_Mi_c_in_F_q^l}), and
    $$L(c,d)\star L(c,d)=(c,d)*(c,d)=(2c^2+2\alpha d^{\sigma(p^k+1)},2cd).$$

    Notice that the middle nucleus is a finite field in the semifield, hence we have $c^2+\alpha d^{\sigma(p^k+1)}\in\mathbb{F}_{p^l}$, which means that $d^{p^k+1}\in\mathbb{F}_{p^l}$. Since $l=\gcd(m,k)$, we have $d^{p^k+1}=d^{p^{2k}+p^k}$, hence $d\in \mathbb{F}_{p^{k}}\cap\mathbb{F}_{p^m}=\mathbb{F}_{p^l}$.

    Now we see that $N_m(\mathbb{S}_{k,\sigma})$ is isomorphic to a subfield of $\mathbb{F}_{p^{2l}}$. Notice that $\mathbb{F}_{p^l}$ is a subfield of $N_m(\mathbb{S}_{k,\sigma})$, according to the extension of finite field, we only need to prove that it is impossible to have $d=1$. Now let $d=f=1$ and $ce=0$, then by (\ref{eq_S_v}), we have $v=\alpha$. Furthermore, (\ref{eq_Mi_long1}) becomes
    $$\alpha((b^{p^k}+b)^\sigma-2b^{p^k})-\alpha(b-b^{p^k})=0,$$
    which means that
    $$(b^{p^k}+b)^\sigma=(b^{p^k}+b),$$
    holds for all $b\in \mathbb{F}_{p^m}$. However, this cannot hold, since $p^k\cdot p^s<p^m$, where $x^\sigma=x^{p^s}$. Hence, $d=0$. \qed
\end{proof}

Noticing that $N_l(\mathbb{S}_{k,\sigma})=N_r(\mathbb{S}_{k,\sigma})=N(\mathbb{S}_{k,\sigma})$, since $\mathbb{S}_{k,\sigma}$ is commutative, the nucleus of $\mathbb{S}_{k,\sigma}$ can also be derived:
\begin{theorem}\label{th_S_Nl}
    Let $\mathbb{S}_{k,\sigma}$ be the semifield with the multiplication $\star$ defined as in (\ref{eq_mola_semifield_multi}) on $\mathbb{F}_{p^{2m}}$, then its (left, right) nucleus $N(\mathbb{S}_{k,\sigma})$ is isomorphic to $\mathbb{F}_{p^{t}}$, where $x^\sigma=x^{p^s}$ and $t=\gcd(m,k,s)$.
\end{theorem}
\begin{proof}
    By using the same notations as in the proof of Theorem \ref{th_S_Nm}, assume that $(a,b)$ is an element in $N(\mathbb{S}_{k,\sigma})$. Since $N(\mathbb{S}_{k,\sigma})\subseteq N_m(\mathbb{S}_{k,\sigma})$, by Theorem \ref{th_S_Nm}, we have $a\in \mathbb{F}_{p^l}=\mathbb{F}_{p^k}\cap\mathbb{F}_{p^m}$ and $b=0$. Moreover, by (\ref{eq_Mi_(xy)z}) and (\ref{eq_Mi_x(yz)}), we have
    $$    L^{-1}((a,b)*(c,d))*(e,f)
        =(u\circ_ke+\alpha(f\circ_k(ad))^\sigma, uf+ade),$$
    and
    $$    (a,b)*L^{-1}((c,d)*(e,f))
        =(a\circ_kv, a(cf+de))$$
    Since $u+u^{p^k}=a\circ_k c+\alpha (b\circ_k d)^\sigma=a(c+c^{p^k})$, we have $u=ac$ and
    $$L^{-1}((a,b)*(c,d))*(e,f)=(a(c\circ_ke)+\alpha a^\sigma(f\circ_k d)^\sigma, acf+ade),$$
    $$(a,b)*L^{-1}((c,d)*(e,f))=(a(v+v^{p^k}), a(cf+de)).$$
    By the definition of $v$, it follows that: $$L^{-1}((a,b)*(c,d))*(e,f)=(a,b)*L^{-1}((c,d)*(e,f)) \hbox{ if and only if } a^\sigma=a.$$
    Since $N(\mathbb{S}_{k,\sigma})\subseteq N_m(\mathbb{S}_{k,\sigma}) $, when $\sigma$ is non-trivial, from Theorem \ref{th_S_Nm} (2), we know that $N_m(\mathbb{S}_{k,\sigma})=\{(a,0)|a\in \mathbb{F}_{p^l}, l=\gcd(m,k)\}$. Therefore, we have $N(\mathbb{S}_{k,\sigma})\simeq \mathbb{F}_{p^t}$.

    When $\sigma$ is the identity mapping, let $a,b\in\mathbb{F}_{p^l}$, $b\neq 0$ and $d=f=0$. Assume that $(a,b)\in N(\mathbb{S}_{k,\sigma})$, then by comparing the second components of (\ref{eq_Mi_(xy)z}) and (\ref{eq_Mi_x(yz)}), we have $v=ce$, which means that $$v+v^{p^k}=ce+c^{p^k}e^{p^k}.$$ However, by (\ref{eq_S_v}), we have
    $$v+v^{p^k}=c\circ_k e.$$
    Hence,
    $$(c-c^{p^k})(e-e^{p^k})=0,$$
    which cannot hold for $c,e\in \mathbb{F}_{p^m}\setminus\mathbb{F}_{p^k}$. Hence, $N(\mathbb{S}_{k,\sigma})$ is a proper subset of $N_m(\mathbb{S}_{k,\sigma})$.

    Furthermore, since $N(\mathbb{S}_{k,\sigma})$ is a subfield in the finite field $N_m(\mathbb{S}_{k,\sigma})\cong\mathbb{F}_{p^{2l}}$, and it is routine to show that $\mathbb{F}_{p^l}\subseteq N(\mathbb{S}_{k,\sigma})$, finally we have $N(\mathbb{S}_{k,\sigma})\cong\mathbb{F}_{p^{l}}$ and $l=t=\gcd(m,k,0)$. \qed
\end{proof}

\begin{remark}
    If we let $k=0$, then $\mathbb{S}_{k,\sigma}$ is a Dickson semifield. It is easy to see that Theorem \ref{th_S_Nm} and Theorem \ref{th_S_Nl} also hold for Dickson semifields.
\end{remark}

\section{The isotopism between $\mathbb{S}_{k,\sigma}$}
It is natural to ask that for the same $m$ but different $k$ and $\sigma$, whether (\ref{eq_mola_multi}) define isotopic presemifields. As we mentioned after Theorem \ref{th_Mola_presemifield}, if there exists some $u$ such that $p^k+1=p^u(p^s+1)\mod p^m$, then $\mathbb{P}_{k,\sigma}$ is isotopic to $\mathbb{P}_{s,\sigma}$; for different $\sigma$ but same $k$, the same result can also be derived. Furthermore, we can prove the following:
\begin{theorem}\label{th_P_strong_isotopism}
    Let $\mathbb{P}_{k,\sigma}$ be the presemifield with the multiplication $*$ defined as in (\ref{eq_mola_multi}) on $\mathbb{F}_{p^{2m}}$. Let $0< k, s\le \lfloor\frac{m}{2}\rfloor$ and $0\le r, t\le \lfloor\frac{m}{2}\rfloor$, where $\sigma(x)=x^{p^r}$ and $\tau(x)=x^{p^t}$. If $(k,\sigma)\neq(s, \tau)$, then $\mathbb{P}_{k,\sigma}$ and $\mathbb{P}_{s,\tau}$ are not strongly isotopic.
\end{theorem}
\begin{proof}
    Denote $f$ and $g$ as the corresponding planar functions from $\mathbb{P}_{k,\sigma}$ and $\mathbb{P}_{s,\tau}$.
    Since strong isotopism between $\mathbb{P}_{k,\sigma}$ and $\mathbb{P}_{s,\tau}$ is equivalent to the linear equivalence between $f$ and $g$, we assume that there exist linearized polynomials $L_1(x,y), L_2(x,y)$ and $L(x,y)$ with both $L(x,y)$ and $(L_1(x,y), L_2(x,y))$ invertible, such that
\begin{equation}\label{eq_P_strong_isotopism}
    L\left(L_1(x,y)^{p^k+1}+\alpha (L_2(x,y)^{p^k+1})^\sigma, L_1(x,y)L_2(x,y)\right)=(x^{p^s+1}+\alpha (y^{p^s+1})^\tau, xy)
\end{equation}
    For convenience, we denote $L_i(x,0)$ and $L_i(0,y)$ by $L_i(x)$ and $L'_i(y)$ respectively. We first prove that:\\
    \textbf{Claim 1:} If (\ref{eq_P_strong_isotopism}) holds, then $s=k$ and $L_i(x)$ and $L'_i(y)$ are monomials or zero, for $i=1,2$.

    Here we only prove the results for $L_i(x)$, by the symmetry, a similar proof can be derived for $L'_i(y)$.
    Let $y=0$, we have,
    \begin{align*}
        &\left(L_1(x)^{p^k+1}+\alpha (L_2(x)^{p^k+1})^\sigma, L_1(x)L_2(x)\right)\\
        =&L^{-1}(x^{p^s+1}, 0)\\
        =&(\varphi_1(x^{p^s+1}), \varphi_2(x^{p^s+1})),
    \end{align*}
    where $L_1(x)=\sum_{i=0}^{m-1}a_i x^{p^i}$, $L_2(x)=\sum_{i=0}^{m-1}b_i x^{p^i}$, $\varphi_1(x)=\sum_{i=0}^{m-1}c_i x^{p^i}$ and $\varphi_2(x)=\sum_{i=0}^{m-1}d_i x^{p^i}$ are linearized polynomials.
   We divide the following proof into two cases:
    \begin{enumerate}
      \item Neither of $L_1(x)$ or $L_2(x)$ equals $0$;
      \item $L_1(x)$ or $L_2(x)$ equals $0$.
    \end{enumerate}

    \textbf{Case(1):} Since $L_1(x)L_2(x)=\varphi_2(x^{p^s+1})$ and $s>0$, we have
    $$\left\{
      \begin{array}{ll}
        (a_ib_{i+s}+a_{i+s}b_i)=d_i, & \hbox{for any $i$;} \\
        a_ib_j+a_jb_i=0, & \hbox{$j\neq i\pm s$.}
      \end{array}
    \right.$$
    Assume that $d_u\neq 0$, then noticing that $a_ib_i=0$ for any $0\le i\le m-1$, we have $d_u=a_ub_{u+s}$ or $a_{u+s}b_u$.

    \textbf{(a)} If $a_u\neq 0$, then $b_u=0$ and for any $j\neq u\pm s$, we have
    $$a_u b_j+a_j b_u=0,$$
    which means that $b_j=0$, and $L_2(x)=b_{u-s}x^{p^{u-s}}+b_{u+s}x^{p^{u+s}}$. If $b_{u\pm s}\neq 0$, then we can also use a similar argument to prove that $L_1(x)=a_{u}x^{p^u}$. Furthermore, we have
    \begin{equation}\label{eq_P_strIso_1}
        \varphi_1(x^{p^s+1})=a_{u}^{p^k+1}(x^{p^k+1})^{p^{u}}+\alpha (b_{u-s}x^{p^{u-s}}+b_{u+s}x^{p^{u+s}})^{(p^k+1)\sigma}.
    \end{equation}
    The right side of (\ref{eq_P_strIso_1}) is
    \begin{align*}
        &\left(a_u^{p^k+1}x^{(p^k+1)p^u}+\alpha b_{u-s}^{(p^k+1)\sigma}x^{(p^k+1)\sigma p^{u-s}}+\alpha b_{u+s}^{(p^k+1)\sigma}x^{(p^k+1)\sigma p^{u+s}}\right)\\
        +&\alpha\left(b_{u-s}^{p^k}b_{u+s}x^{(p^{k-s}+p^s) p^u}+b_{u-s}b_{u+s}^{p^k}x^{(p^{-s}+p^{s+k}) p^u}\right)^{\sigma}
    \end{align*}
     which means that (\ref{eq_P_strIso_1}) does not hold, since $x^{(p^{k-s}+p^s)}$ and $x^{(p^{-s}+p^{s+k})}$ cannot be simultaneously written in the form $x^{(p^s+1)p^i}$ for some $i$ respectively. Therefore, one of $b_{u-s}$ and $b_{u+s}$ must be $0$.

     If $b_{u-s}=0$, then we can derive that $L_1(x)=a_{u}x^{p^u}+a_{u+2s}x^{p^{u+2s}}$. Similarly as (\ref{eq_P_strIso_1}), it can also be proved that $a_{u+2s}=0$ by symmetry basically, and we see that $L_1(x)$ and $L_2(x)$ are both monomials, and we have that
    \begin{equation}\label{eq_P_strIso_2}
        \varphi_1(x^{p^s+1})=a_{u}^{p^k+1}(x^{p^k+1})^{p^{u}}+\alpha (b_{u+s}x^{p^{u+s}})^{(p^k+1)\sigma}.
    \end{equation}
    When $s\neq k$, then (\ref{eq_P_strIso_2}) also cannot hold, otherwise  that means the presemifields on $\mathbb{F}_{p^m}$ defined by $x^{p^s+1}$ and $x^{p^k+1}$ are isotopic.

    If $b_{u+s}\neq 0$, then by symmetry we can get $L_1(x)=a_{u+s}x^{p^{u+s}}$ and $L_2(x)=b_{u}x^{p^u}$ and $s=k$.

    \textbf{(b)} Similarly as in \textbf{(a)}, if $b_u\neq 0$, by the symmetry of $L_1$ and $L_2$ in $L_1(x)L_2(x)=\varphi_2(x^{p^s+1})$, we can prove that $s=k$ and both $L_1(x)$ and $L_2(x)$ are monomials.

    \textbf{Case(2):}
    Assume that $L_1(x)\neq 0$ and $L_2(x)=0$, w.l.o.g, we have $L_1(x)^{p^k+1}=\varphi_1(x^{p^s+1})$. It cannot hold for $s\neq k$, since two generalized twisted fields are not isotopic. When $s=k$, according to Theorem 5.2 in \cite{Biliotti1999}, we know that $L_1(x)$ and $\varphi_1(x)$ are both linearized monomials.

    Therefore, we have proved the \textbf{Claim 1}.

    Now, for $k=s$, we know that $L_1(x,y)$ and $L_2(x,y)$ are both linearized binomials or monomials. Assume that the possible degrees of $x$ in $L_1$ and $L_2$ are $p^u$ and $p^{u+k}$, those of $y$ are $p^v$ and $p^{v+k}$, then there are four possible combinations of them to form $L_1$ and $L_2$:
    \begin{enumerate}
      \item $L_1:(x^{p^u}, y^{p^v}), L_2:(x^{p^{u+k}}, y^{p^{v+k}})$;
      \item $L_1:(x^{p^u}, y^{p^{v+k}}), L_2:(x^{p^{u+k}}, y^{p^v})$;
      \item $L_1:(x^{p^{u+k}}, y^{p^v}), L_2:(x^{p^u}, y^{p^{v+k}})$;
      \item $L_1:(x^{p^{u+k}}, y^{p^{v+k}}), L_2:(x^{p^u}, y^{p^v})$.
    \end{enumerate}
    First, let us assume that $L_1$ and $L_2$ are both binomials. In fact, noticing that there is only $xy$ on the right side of (\ref{eq_P_strong_isotopism}), both $(1)$ and $(4)$ are not feasible, and there must be $u=v$ for $(2)$ and $(3)$. Hence, firstly we consider case (2): $L_1(x,y)=a_{u}x^{p^{u}}+a'_{u+k}y^{p^{u+k}}$ and $L_2(x,y)=b_{u+k}x^{p^{u+k}}+b'_{u}y^{p^{u}}$. Let
    $$L^{-1}=\left(
                  \begin{array}{cc}
                    \varphi_1 & \varphi_3 \\
                    \varphi_2 & \varphi_4 \\
                  \end{array}
                \right),
    $$
    then by expanding the first component of (\ref{eq_P_strong_isotopism}),
    we have
    $$L_1(x,y)^{p^k+1}+\alpha (L_2(x,y)^{p^k+1})^\sigma=\varphi_1(x^{p^k+1}+\alpha  (y^{p^k+1})^\tau)+\varphi_3(xy),$$
    The terms are $a_ux^{p^u}a_{u+k}'^{p^k}y^{p^{u+2k}}$ and $(b_{u+k}^{p^k}x^{p^{u+2k}}b'_uy^{p^u})^\sigma$ occur on the left side, but they are impossible to appear on the right side. That means $L_1(x)$ and $L_2(x)$ are both monomials with the same degree.

    First, we assume that $a'_{u+k}=b_{u+k}=0$ and we have
    $$(a_ux^{p^u})^{p^k+1}+\alpha ((b'_uy^{p^u})^{p^k+1})^{\sigma}=\varphi_1(x^{p^k+1}+\alpha  (y^{p^k+1})^\tau)+\varphi_3(xy).$$
    Obviously, $\varphi_1$ is a monomial and $\varphi_3\equiv0$. Therefore, one necessary condition for the upper equation hold is that $\sigma=\tau$.

    Secondly, by assuming that $a'_{u}=b_{u}=0$ and $w=u+k$, we have that $L_1(x,y)=a'_{w}y^{p^{w}}$, $L_2(x,y)=b_{w}x^{p^{w}}$, and
    \begin{equation*}
      (a'_wy^{p^w})^{p^k+1}+\alpha ((b_wx^{p^w})^{p^k+1})^{\sigma}=\varphi_1(x^{p^k+1}+\alpha  (y^{p^k+1})^\tau)+\varphi_3(xy).
    \end{equation*}
    It is also easy to see that $\varphi_1(x)$ is a monomial and $\varphi_3(x)=0$. Hence, one necessary condition that the upper equation holds is $\sigma=\tau$ and $\sigma^2$ is identity.

    For case (3), we can derive the same result as case (2).
    Hence, this completes the proof.\qed
\end{proof}

To investigate the isotopism between $\mathbb{P}_{k,\sigma}$ and $\mathbb{P}_{s,\tau}$ further, we need the following results from \cite{Coulter2008}:
\begin{theorem}(Coulter and Henderson)\label{th_CoulterHenderson}
    Let $F_1 =(\mathbb{F}_q, +, \star)$ and $F_2 =(\mathbb{F}_q, +, *)$ be isotopic commutative semifields. Then every isotopism $(M, N, L)$ between $F_1$ and $F_2$ satisfies either
    \begin{enumerate}[(1)]
      \item $M=N$, or
      \item $M(x)\equiv \gamma\star N(x) \mod (x^q-x)$, where $\gamma \in N_m(F_1)$.
    \end{enumerate}
\end{theorem}

By replacing $N(x)$ by $x'$ and $N(y)$ by $y'$ in Theorem \ref{th_CoulterHenderson}(2), we have $(\gamma\star x')\star y'=L(N^{-1}(x')*N^{-1}(y'))$. Since we have considered the strong isotopism between $\mathbb{P}_{k,\sigma}$ and $\mathbb{P}_{s,\tau}$ already, we only need to pick up some elements in $N_m(\mathbb{S}_{k,\sigma})$ and $N_m(\mathbb{S}_{s,\tau})$ respectively, and take $N$ as the identity mapping to investigate case $(2)$ in Theorem \ref{th_CoulterHenderson}.

When $\sigma$ is non-trivial, by Theorem \ref{th_S_Nm}, we know that $N_m(\mathbb{S}_{k,\sigma})=\mathbb{F}_{p^m}\cap \mathbb{F}_{p^k}=\{(c,0)\mid c\in\mathbb{F}_{p^l} \}$ with  $l=\gcd(k,m)$, and
$$L^{-1}\left((a,b)*(c,0)\right)=L^{-1}\left(a\circ_k c, bc\right)=(ac,bc).$$
Therefore, we can define a new semifield multiplication as
\begin{align*}
  &(a,b)\otimes (e,f)\\
  =&L^{-1}\left((a,b)*(c,0)\right)*(e,f)\\
  =&  \left((a\circ_k e)c+\alpha c^\sigma(b\circ_k f)^\sigma, c(af+be)\right),
\end{align*}
and the corresponding planar function $f$ can be written as
$$f(x,y)=2\cdot(cx^{p^k+1}+\alpha c^{\sigma}(y^{p^k+1})^\sigma, cxy),$$
which is equivalent to the one defined by $\mathbb{P}_{k,\sigma}$, when $c\neq 0$. Therefore, by Theorem \ref{th_P_strong_isotopism}, we have:
\begin{corollary}\label{cr_strong_iso_Pk_sigma}
    Let $\mathbb{P}_{k,\sigma}$ be the presemifield with the multiplication $*$ defined as in (\ref{eq_mola_multi}) on $\mathbb{F}_{p^{2m}}$ with $\sigma$ non-trivial. Let $0< k, s\le \lfloor\frac{m}{2}\rfloor$ and $0\le r, t\le \lfloor\frac{m}{2}\rfloor$, where $\sigma(x)=x^{p^r}$ and $\tau(x)=x^{p^t}$. If $(k,\sigma)\neq(s, \tau)$, then $\mathbb{P}_{k,\sigma}$ and $\mathbb{P}_{s,\tau}$ are not isotopic.
\end{corollary}

For the case that $\sigma$ is trivial, we can define a commutative presemifield $\tilde{\mathbb{P}}_k$ by the following multiplication:
$$x\otimes y =L^{-1}(x*\beta)*y,$$
where $L(a,b)=(a,b)*(1,0)=(a+a^{p^k},b)$ and $L(\beta)\in N_m(\mathbb{S}_{k})$. As proved in Theorem \ref{th_S_Nm}, $\beta=(c,d)\in \mathbb{F}_{p^{l}}^2$, where $l=\gcd(m,k)$, and we have
\begin{align*}
    &L^{-1}\left((a,b)*(c,d)\right)\\
    =&L^{-1}\left(c(a+a^{p^k})+\alpha d(b+b^{p^k}), ad+bc)\right)\\
    =&(ac+\alpha bd,ad+bc),
\end{align*}
which means that the multiplication of $\tilde{\mathbb{P}}_k$ can be written as
\begin{align*}
    &(a,b)\otimes(e,f)\\
    =&\left(c(a\circ_k e+b\circ_k f\alpha)+\alpha d(b\circ_k e+a\circ_k f), c(af+be)+d(ae+bf\alpha)\right).
\end{align*}
According to Theorem \ref{th_S_Nl}, if $d=0$, then $\beta=(c,d)$ will correspond to some element in $N_l(\mathbb{S}_{k})$, which leads to strong isotopism. Without loss of generality, assume that $d=1$,
we can define $(a,b)\diamond(e,f)$ as
\begin{equation}\label{eq_Mi_multi}
    \left(c(a\circ_k e+b\circ_k f\alpha)+\alpha (b\circ_k e+a\circ_k f), c(af+be)+(ae+bf\alpha)\right),
\end{equation}
and the corresponding planar function $f$ is
$$f(x,y)=(2c(x^{p^k+1}+\alpha y^{p^k+1})+2\alpha x\circ_k y, 2cxy+x^2+\alpha y^2),$$
which is equivalent to
\begin{equation}\label{eq_mola_PNF_2}
    f(x,y)=(2cxy+x^2+\alpha y^2, c(x^{p^k+1}+\alpha y^{p^k+1})+\alpha x\circ_k y).
\end{equation}

Similarly as the proof of Theorem \ref{th_P_strong_isotopism}, we can get the following results:
\begin{theorem}\label{th_P_time_r}
    If $c^2-\alpha$ is a non-square in $\mathbb{F}_{p^l}$, where $l=\gcd(m,k)$, then the presemifield defined by $\diamond$ in (\ref{eq_Mi_multi}) is not strongly isotopic with $\mathbb{P}_{s}$, for any $s>0$.
\end{theorem}
\begin{proof}
    Assume that the two presemifield are strongly isotopic, then we have linearized polynomials $L_1(x,y), L_2(x,y)$ and $L(x,y)$, where $L(x,y)$ is a permutation such that
    \begin{align*}
        &L\circ \left(
                  \begin{array}{c}
                    2cL_1(x,y)L_2(x,y)+L_1(x,y)^{2}+\alpha L_2(x,y)^{2} \\
                    c(L_1(x,y)^{p^k+1}+\alpha L_2(x,y)^{p^k+1})+\alpha L_1(x,y)\circ_kL_2(x,y) \\
                  \end{array}
                \right)^T\\
        =&(x^{p^s+1}+\alpha y^{p^s+1}, xy).
    \end{align*}
    Let $y=0$, and we denote $L_i(x,0)$ by $L_i(x)$, for convenience. We get,
    \begin{align*}
    &\left(
        \begin{array}{c}
          2cL_1(x)L_2(x)+L_1(x)^2+\alpha L_2(x)^2 \\
          c(L_1(x)^{p^k+1}+\alpha L_2(x)^{p^k+1})+\alpha L_1(x)\circ_kL_2(x) \\
        \end{array}
     \right)^T\\
    =&L^{-1}(x^{p^s+1}, 0)=(\varphi_1(x^{p^s+1}), \varphi_2(x^{p^s+1})),
    \end{align*}
    where $\varphi_1(x), \varphi_2(x)$ are linearized polynomials. Let $L_1(x)=\sum_{i=0}^{m-1}a_ix^{p^i}$, $L_2(x)=\sum_{i=0}^{m-1}b_ix^{p^i}$ and $\varphi_1(x)=\sum_{i=0}^{m-1}c_ix^{p^i}$, then
    \begin{align*}
        &L_1(x)^2+\alpha L_2(x)^2 + 2cL_1(x)L_2(x)\\
        =& \sum_{i>j}2(a_ia_j+\alpha b_ib_j+ca_ib_j+ca_jb_i)x^{p^i+p^j}+\sum_{i=0}^{m-1}(a_i^2+\alpha b_i^2+2ca_ib_i)x^{2p^i}.
    \end{align*}
    Since $s\neq 0$, by comparing the equation above with $\varphi_1(x^{p^s+1})$, we have that
    $$a_i^2+\alpha b_i^2+2ca_ib_i=0, \hbox{  for any } 0\le i\le m-1,$$
    which can also be written as,
    $$(a_i+c b_i)^2+(\alpha-c^2)b_i^2=0, \hbox{  for any } 0\le i\le m-1.$$
    If $c^2-\alpha$ is a non-square in $\mathbb{F}_{p^l}$, then it is also a non-square in $\mathbb{F}_{p^m}$, since $\frac{m}{l}$ is odd.  Hence the equation above has no solution. Therefore, we prove the claim in our theorem. \qed
\end{proof}

It is well-known that there always exist some $c\in\mathbb{F}_{p^l}$ such that $c^2-\alpha$ is a non-square in $\mathbb{F}_{p^l}$, where $l=\gcd(m,k)$ and $\alpha\in\mathbb{F}_{p^l}$ is also a non-square. Therefore, from Theorem \ref{th_P_strong_isotopism} and Theorem \ref{th_P_time_r}, we have

\begin{corollary}
    For every $k$, the semifield $\mathbb{S}_{k}$ defines two inequivalent planar functions over $\mathbb{F}_{p^{2m}}$.
\end{corollary}

The total number of non-isotopic semifields and inequivalent planar functions defined by $\mathbb{S}_{k,\sigma}$ can also be counted:
\begin{corollary}\label{cr_count_semi}
    Let $\mathbb{S}_{k,\sigma}$ be the semifield with the multiplication $\star$ defined as in (\ref{eq_mola_multi}) on $\mathbb{F}_{p^{2m}}$, where $m=2^e\mu$ with $\gcd(\mu, 2)=1$, then  $\mathbb{S}_{k,\sigma}$  defines
    \begin{enumerate}
      \item $\lfloor\frac{\mu}{2}\rfloor\cdot\lceil\frac{m}{2}\rceil$ non-isotopic semifields;
      \item $\lfloor\frac{\mu}{2}\rfloor\cdot(\lceil\frac{m}{2}\rceil+1)$ inequivalent planar functions.
    \end{enumerate}
\end{corollary}

\section{$\mathbb{S}_{k,\sigma}$ is a new family}
In the previous sections, we showed that our new family looks like a combination of Dickson semifields and generalized twisted fields, and $\mathbb{S}_{k}$ behaves quite different from $\mathbb{S}_{k,\sigma}$ with nontrivial $\sigma$. Therefore, we suggest to divide it into two families, according to whether $\sigma$ is trivial, as the case of finite fields and Dickson semifields.

When we take them as two familes, then one natural question is:

\emph{Do $\mathbb{S}_{k}$ and $\mathbb{S}_{k,\sigma}$ contain new semifields compared with the other known families? }

In fact, for some cases, we can prove that $\mathbb{S}_{k}$ is contained in the family discovered by Budaghyan and Helleseth \cite{B-H2008,BudaghyanHelleseth2010}, which can be rewritten in the following form:
\begin{theorem}[Kyureghyan and Bierbrauer \cite{BierbrauerKyureghyan2010}]\label{th_BHB_Gohar}
    Let $p$ be an odd prime number, $q=p^m$, $n=2m$ and integers $i$, $j$ such that $s=i-j$. Then the mapping $M_s: \mathbb{F}_{p^n} \rightarrow \mathbb{F}_{p^n}$ given by
    $$M_s(x)=x^{p^m+1}+\omega \mathrm{tr}_{q^2/q}(\beta x^{p^i+p^j}),~~~~~~~~i\ge j\ge 0,$$
    is planar if and only if all the following conditions are fulfilled:
    \begin{enumerate}
      \item s=0 or $\nu (s)\neq \nu(m)$,
      \item $\omega \in \mathbb{F}_{q^2}\setminus \mathbb{F}_q$,
      \item $\beta$ is a non-square in $\mathbb{F}_{q^2}$,
    \end{enumerate}
    where $s=2^{\nu(s)}s_1$ with $s_1$ an odd integer.
\end{theorem}

Since $u^{p^m+1}\in\mathbb{F}_{p^m}$, for any $u\in \mathbb{F}_{p^{2m}}$, different choices of $\omega$ give equivalent planar functions. Furthermore, it is easy to see that different $(i,j)$ with the same $s$ also lead to equivalent $M_s$, so we redefine $M_s(x)$ as follows:
\begin{equation}\label{eq_BHB-2}
    M_s(x)=x^{p^m+1}+\omega \mathrm{tr}_{q^2/q}(\beta x^{p^s+1}),
\end{equation}
where $s=0$ or $\nu(s)\neq \nu(m)$, and $\beta$ and $\omega$ are the same as in Theorem \ref{th_BHB_Gohar}.

The following lemma can be found in \cite{Coulter1998}, \cite{DraperHou2007} and \cite{Helleseth2007}:
\begin{lemma}\label{le_gcd_pj+1_pn-1}
    For an odd prime $p$,
    $$\gcd(p^j+1, p^n-1)=\left\{
                           \begin{array}{ll}
                             p^{\gcd(j,n)}+1, & \hbox{if $\nu(j)<\nu(n)$;} \\
                             2, & \hbox{otherwise.}
                           \end{array}
                         \right.
    $$
\end{lemma}

When $m$ is odd, there exist $\omega\in \mathbb{F}_{p^2}\setminus\mathbb{F}_p$, such that $\omega+\omega^{p^m}=0$ and $u\in\mathbb{F}_{p^{2m}}$ can be written as $a+b\omega$, where $a,b\in \mathbb{F}_{p^m}$. Furthermore, by Lemma \ref{le_gcd_pj+1_pn-1}, since $m$ is odd and $\nu(s)\neq \nu(m)$, we have $\gcd(p^s+1, p^{2m}-1)=2$, so any different non-square $\beta$ lead to equivalent $M_s(x)$ and we assume that $\beta=\omega^{-1}$.  Then we denote $\odot$ to be the multiplication defined by (\ref{eq_BHB-2}), denote $u,v \in\mathbb{F}_{p^{2m}}$ respectively by $a+b\omega$ and $c+d\omega$, and we have
\begin{align*}
    u\odot v=&(a+b\omega)\odot(c+d\omega)\\
    =&(a+b\omega^{p^m})(c+d\omega)+(a+b\omega)(c+d\omega^{p^m})+\omega \mathrm{tr}_{q^2/q}(\omega^{-1} x\circ_s Y)\\
    =&2ac-2bd\omega^2+((a+b\omega)\circ_s(c+d\omega)-(a-b\omega)\circ_s(c-d\omega))\\
    =&2(ac-bd\omega^2)+2(a\circ_s (d\omega)+(b\omega)\circ_s c)\\
    =&2(ac-bd\omega^2)+2(a\circ_s d+ b\circ_s c)\omega.
\end{align*}
The last equality sign holds, because $\omega^{p^s-1}=1$, since that $s$ must be even.
Moreover, the corresponding planar function is equivalent to
$$M_s(x,y)=(x^{p^s}y+xy^{p^s},x^2-\omega^2y^2),$$
which is equivalent to (\ref{eq_mola_PNF_2}) with $c=0$, when $-1$ is a square in $\mathbb{F}_{p^m}$.

However, on the other hand, since we showed above that when $m$ is odd and $-1$ is a square, the Budaghyan-Helleseth semifield is isotopic to $\mathbb{S}_k$, it can not be isotopic to $\mathbb{S}_{k,\sigma}$ with non-trivial $\sigma$ by Corollary \ref{cr_strong_iso_Pk_sigma}. Furthermore, by the middle and left nucleus of $\mathbb{S}_{k,\sigma}$, we know that it is not isotopic with Albert and Dickson semifields. Moreover, since $\mathbb{S}_{k,\sigma}$ is defined over $p^{2m}$ with any odd $p$, it must cover some semifields which are not included by the known semifields except for the Lunardon-Marino-Polverino-Tormbetti semifields generalized  by Bierbrauer in \cite{Bierbrauer2009-3} (abbreviated to LMPTB). Nevertheless, for given $p$ and $m$, LMPTB has only one element, but by Corollary \ref{cr_count_semi} our family $\mathbb{S}_{k,\sigma}$ has more when $m\ge 5$. Therefore, to sum up,
\begin{theorem}
    When $m\ge 5$ is odd and $-1$ is a square, then $\mathbb{S}_{k,\sigma}$ with non-trivial $\sigma$ contains semifields which are not isotopic with any known ones.
\end{theorem}

\section{APN functions with the similar form}

In \cite{Zha2009}, Zha, Kyureghyan and Wang showed for the first time that some planar functions can be derived from quadratic APN functions. Similar constructions for planar function can also be found in \cite{Bierbrauer2009-1,BierBrauer2009-2,B-H2008,Zha2009-1}. One natural question is the following: Is it possible to get some new APN functions from known planar ones?

In fact, from our new presemifields family, we can derive a similar family of APN functions on $\mathbb{F}_{2^{2m}}$:
\begin{theorem}\label{th_Mola_APN}
    Let $m\ge 2$ be even integer, and k be a integer such that $\gcd(k,m)=1$. Define a function $f$ on $\mathbb{F}_{2^{2m}}$ as follows,
    \begin{equation*}
        f(x,y)=(x^{2^k+1}+\alpha y^{(2^k+1)\sigma}, xy),
    \end{equation*}
    where $\alpha\in\mathbb{F}_{2^m}$, $\alpha\neq 0$ and $\sigma\in\mathrm{Aut}(\mathbb{F}_{2^m})$.
    Then $f$ is an APN function, if and only if $\alpha$ cannot be written as $a^{2^k+1}(t^{2^k}+t)^{1-\sigma}$, where $a,t\in\mathbb{F}_{2^m}$.
\end{theorem}
\begin{proof}
    Since $f$ is quadratic, we only have to prove that for each $(a,b)\neq 0$, equations
    $$\left\{
    \begin{array}{l}
        x\circ_k a+\alpha (y\circ_k b)^\sigma=0 \\
        ay+bx=0
      \end{array}\right.
    $$
    have at most two roots, where $x\circ_k y=x^{p^k}y+y^{p^k}x$.

    If $b=0$, then we have $x\circ_k a=0$ and $ay=0$, which means $y=0$, $x=0$ or $a$, since $x^{p^k+1}$ is APN function on $\mathbb{F}_{2^m}$ and $a\neq 0$.

    If $b\neq 0$, then $x =\frac{ay}{b}=t\cdot a$. where $t :=\frac{y}{b}$. Hence, we have
    $$(at)\circ_k a+\alpha ((bt)\circ_k b)^\sigma=0,$$
    namely,
    $$(t^{2^k}+t)a^{2^k+1}+\alpha(t^{2^k}+t)^\sigma b^{(2^k+1)\sigma}=0.$$
    If $t^{2^k}+t=0$, then $x=y=0$ or $y=b$, $x=a$; If $t^{2^k}+t\neq 0$, then we have
    $$\alpha=\left(\frac{a}{b^\sigma}\right)^{2^k+1}(t^{2^k}+t)^{1-\sigma},$$
    from which we prove the claim. \qed
\end{proof}

Moreover, let us further consider the condition of Theorem \ref{th_Mola_APN}. For even $m$, when $\gcd(k,m)=1$, we have $\gcd(2^k+1, 2^m-1)=3$ and $\gcd(2^i-1, 2^m-1)=2^{\gcd(i,m)}-1$. Hence, if $i$ is even and $\sigma(x)=x^{2^i}$, then $a^{2^k+1}(t^{2^k}+t)^{1-\sigma}$ is a cubic. Therefore, if $\alpha$ is not a cubic, then the condition in Theorem \ref{th_Mola_APN} holds.
\begin{corollary}\label{cr_APN}
    Let $m\ge 2$ be even integer, and $k$ be a integer such that $\gcd(k,m)=1$. Define a function $f$ on $\mathbb{F}_{2^{2m}}$ as follows:
    \begin{equation*}
        f(x,y)=(x^{2^k+1}+\alpha y^{(2^k+1){2^i}}, xy),
    \end{equation*}
    where the nonzero $\alpha\in\mathbb{F}_{2^m}$ is a non-cubic and $i$ is even, then $f$ is an APN function.
\end{corollary}

Let $m=4$, $k=1$ and $\alpha$ be a primitive element of $\mathbb{F}_{2^4}$, by Corollary \ref{cr_APN}, we can choose $i=0$ or $2$ to get two APN functions. By using MAGMA\cite{Magma}, it can be computed that, when $i=0$, the APN function is equivalent to the function No 2.1 in Table 10 in \cite{EdelPott2009}. However, when $i=2$, the $\Gamma$-rank of the APN function is $13642$, which does not occur in the list of known APN functions in \cite{EdelPott2009} (see \cite{EdelPott2009} for the $\Gamma$-rank). More concretely, the function
$$f(x,y)=(x^3+\alpha y^{12}, xy)$$
is a new APN function on $\mathbb{F}_{2^8}$.

\begin{remark}
    In \cite{Carlet_APN_2010}, Carlet presents some interesting constructions of APN functions, which include a similar result as Theorem \ref{th_Mola_APN} with $\sigma=\mathrm{id}$.
\end{remark}





\bibliographystyle{model1b-num-names}
\bibliography{ref-sc-8-8}

\begin{thebibliography}{37}
\expandafter\ifx\csname natexlab\endcsname\relax\def\natexlab#1{#1}\fi
\providecommand{\bibinfo}[2]{#2}
\ifx\xfnm\relax \def\xfnm[#1]{\unskip,\space#1}\fi
\bibitem[{Albert(1952)}]{Albert52}
\bibinfo{author}{A.~Albert}, \bibinfo{title}{On nonassociative division
  algebras}, \bibinfo{journal}{Transaction of the American Mathematical
  Society} \bibinfo{volume}{72} (\bibinfo{year}{1952})
  \bibinfo{pages}{292--309}.
\bibitem[{Albert(1960)}]{Albert1960}
\bibinfo{author}{A.~Albert}, \bibinfo{title}{Finite division algebras and
  finite planes}, in: \bibinfo{booktitle}{Combinatorial Analysis: Proceedings
  of the 10th Symposium in Appled Mathematics}, volume~\bibinfo{volume}{10} of
  \textit{\bibinfo{series}{Symposia in Appl. Math.}},
  \bibinfo{organization}{American Mathematical Society},
  \bibinfo{address}{Providence, R.I.}, pp. \bibinfo{pages}{53--70}.
\bibitem[{Ball and Lavrauw(2002)}]{BallLavrauw2002}
\bibinfo{author}{S.~Ball}, \bibinfo{author}{M.~Lavrauw},
  \bibinfo{title}{Commutative semifields of rank 2 over their middle nucleus},
  in: \bibinfo{booktitle}{Finite Fields with Applications to Coding Theory,
  Cryptography and Related Areas}, \bibinfo{publisher}{Springer-Verlag},
  \bibinfo{address}{Berlin, New York}, \bibinfo{year}{2002}, pp.
  \bibinfo{pages}{1--21}.
\bibitem[{Bierbrauer(2009{\natexlab{a}})}]{Bierbrauer2009-3}
\bibinfo{author}{J.~Bierbrauer}, \bibinfo{title}{Commutative semifields from
  projection mappings}, \bibinfo{year}{2009}{\natexlab{a}}.
  \bibinfo{note}{Submitted to Design, Codes and Cryptography}.
\bibitem[{Bierbrauer(2009{\natexlab{b}})}]{Bierbrauer2009-1}
\bibinfo{author}{J.~Bierbrauer}, \bibinfo{title}{New commutative semifields and
  their nuclei}, in: \bibinfo{editor}{M.~Bras-Amorós},
  \bibinfo{editor}{T.~Hoholdt} (Eds.), \bibinfo{booktitle}{Applied Algebra,
  Algebraic Algorithms and Error-Correcting Codes}, volume
  \bibinfo{volume}{5527} of \textit{\bibinfo{series}{Lecture Notes in Computer
  Science}}, \bibinfo{publisher}{Springer Berlin / Heidelberg},
  \bibinfo{address}{Tarragona, Spain}, \bibinfo{year}{2009}{\natexlab{b}}, pp.
  \bibinfo{pages}{179--185}.
\bibitem[{Bierbrauer(2010)}]{BierBrauer2009-2}
\bibinfo{author}{J.~Bierbrauer}, \bibinfo{title}{New semifields, {PN} and {APN}
  functions}, \bibinfo{journal}{Des. Codes Cryptography} \bibinfo{volume}{54}
  (\bibinfo{year}{2010}) \bibinfo{pages}{189--200}.
\bibitem[{Bierbrauer and Kyureghyan(2010)}]{BierbrauerKyureghyan2010}
\bibinfo{author}{J.~Bierbrauer}, \bibinfo{author}{G.M. Kyureghyan},
  \bibinfo{title}{On the projection construction of planar and apn mappings},
  \bibinfo{year}{2010}. \bibinfo{note}{Manuscript, appeared at {YACC} 2010,
  Porquerolles Island, France}.
\bibitem[{Biliotti et~al.(1999)Biliotti, Jha and Johnson}]{Biliotti1999}
\bibinfo{author}{M.~Biliotti}, \bibinfo{author}{V.~Jha}, \bibinfo{author}{N.L.
  Johnson}, \bibinfo{title}{The collineation groups of generalized twisted
  field planes}, \bibinfo{journal}{Geometriae Dedicata} \bibinfo{volume}{76}
  (\bibinfo{year}{1999}) \bibinfo{pages}{97--126}.
  \bibinfo{note}{10.1023/A:1005089016092}.
\bibitem[{Bosma et~al.(1997)Bosma, Cannon and Playoust}]{Magma}
\bibinfo{author}{W.~Bosma}, \bibinfo{author}{J.~Cannon},
  \bibinfo{author}{C.~Playoust}, \bibinfo{title}{The {MAGMA} algebra system
  {I}: the user language}, \bibinfo{journal}{J. Symb. Comput.}
  \bibinfo{volume}{24} (\bibinfo{year}{1997}) \bibinfo{pages}{235--265}.
\bibitem[{Budaghyan et~al.(2006)Budaghyan, Carlet and Pott}]{Budaghyan06}
\bibinfo{author}{L.~Budaghyan}, \bibinfo{author}{C.~Carlet},
  \bibinfo{author}{A.~Pott}, \bibinfo{title}{New classes of almost bent and
  almost perfect nonlinear polynomials}, \bibinfo{journal}{IEEE Transactions on
  Information Theory} \bibinfo{volume}{52} (\bibinfo{year}{2006})
  \bibinfo{pages}{1141--1152}.
\bibitem[{Budaghyan and Helleseth(2008)}]{B-H2008}
\bibinfo{author}{L.~Budaghyan}, \bibinfo{author}{T.~Helleseth},
  \bibinfo{title}{New perfect nonlinear multinomials over ${F}_{p^{2k}}$ for
  any odd prime $p$}, in: \bibinfo{booktitle}{SETA '08: Proceedings of the 5th
  international conference on Sequences and Their Applications},
  \bibinfo{publisher}{Springer-Verlag}, \bibinfo{address}{Berlin, Heidelberg},
  \bibinfo{year}{2008}, pp. \bibinfo{pages}{403--414}.
\bibitem[{Budaghyan and Helleseth(2010)}]{BudaghyanHelleseth2010}
\bibinfo{author}{L.~Budaghyan}, \bibinfo{author}{T.~Helleseth},
  \bibinfo{title}{New commutative semifields defined by new {PN} multinomials},
  \bibinfo{journal}{Cryptography and Communications}  (\bibinfo{year}{2010}).
  \bibinfo{note}{Available online}.
\bibitem[{Carlet(2010{\natexlab{a}})}]{Carlet_Vect_Book}
\bibinfo{author}{C.~Carlet}, \bibinfo{title}{Boolean models and methods in
  mathematics, computer science, and engineering}, \bibinfo{title}{Boolean
  Models and Methods in Mathematics, Computer Science, and Engineering}, number
  \bibinfo{number}{134} in \bibinfo{series}{Encyclopedia of Mathematics and its
  Applications}, \bibinfo{publisher}{Cambridge University Press},
  \bibinfo{year}{2010}{\natexlab{a}}, pp. \bibinfo{pages}{398--471}.
\bibitem[{Carlet(2010{\natexlab{b}})}]{Carlet_APN_2010}
\bibinfo{author}{C.~Carlet}, \bibinfo{title}{Relating three nonlinearity
  parameters of vectorial functions and building {APN} functions from bent
  functions}, \bibinfo{year}{2010}{\natexlab{b}}. \bibinfo{note}{To appear in
  Design, Codes and Cryptography}.
\bibitem[{Cohen and Ganley(1982)}]{Cohen-Ganley1982}
\bibinfo{author}{S.~Cohen}, \bibinfo{author}{M.~Ganley},
  \bibinfo{title}{Commutative semifields, two-dimensional over their middle
  nuclei}, \bibinfo{journal}{Journal of Algebra} \bibinfo{volume}{75}
  (\bibinfo{year}{1982}) \bibinfo{pages}{373--385}.
\bibitem[{Coulter(1998)}]{Coulter1998}
\bibinfo{author}{R.S. Coulter}, \bibinfo{title}{Explicit evaluations of some
  weil sums}, \bibinfo{journal}{Acta Arithmetica} \bibinfo{volume}{83}
  (\bibinfo{year}{1998}) \bibinfo{pages}{241--251}.
\bibitem[{Coulter and Henderson(2008)}]{Coulter2008}
\bibinfo{author}{R.S. Coulter}, \bibinfo{author}{M.~Henderson},
  \bibinfo{title}{Commutative presemifields and semifields},
  \bibinfo{journal}{Advances in Mathematics} \bibinfo{volume}{217}
  (\bibinfo{year}{2008}) \bibinfo{pages}{282 -- 304}.
\bibitem[{Coulter et~al.(2007)Coulter, Henderson and
  Kosick}]{Coulter-Henderson-Kosick2007}
\bibinfo{author}{R.S. Coulter}, \bibinfo{author}{M.~Henderson},
  \bibinfo{author}{P.~Kosick}, \bibinfo{title}{Planar polynomials for
  commutative semifields with specified nuclei}, \bibinfo{journal}{Des. Codes
  Cryptography} \bibinfo{volume}{44} (\bibinfo{year}{2007})
  \bibinfo{pages}{275--286}.
\bibitem[{Coulter and Matthews(1997)}]{Coulter-Matthews1997}
\bibinfo{author}{R.S. Coulter}, \bibinfo{author}{R.W. Matthews},
  \bibinfo{title}{Planar functions and planes of {L}enz-{B}arlotti class {II}},
  \bibinfo{journal}{Des. Codes Cryptography} \bibinfo{volume}{10}
  (\bibinfo{year}{1997}) \bibinfo{pages}{167--184}.
\bibitem[{Dembowski and Ostrom(1968)}]{DO1968}
\bibinfo{author}{P.~Dembowski}, \bibinfo{author}{T.~Ostrom},
  \bibinfo{title}{Planes of order $n$ with collineation groups of order $n^2$},
  \bibinfo{journal}{Mathematische Zeitschrift} \bibinfo{volume}{103}
  (\bibinfo{year}{1968}) \bibinfo{pages}{239--258}.
\bibitem[{Dickson(1906)}]{Dickson1906}
\bibinfo{author}{L.~Dickson}, \bibinfo{title}{On commutative linear algebras in
  which division is always uniquely possible}, \bibinfo{journal}{Transaction of
  the American Mathematical Society} \bibinfo{volume}{7} (\bibinfo{year}{1906})
  \bibinfo{pages}{514--522}.
\bibitem[{Ding and Yuan(2006)}]{Ding-Yuan2006}
\bibinfo{author}{C.~Ding}, \bibinfo{author}{J.~Yuan}, \bibinfo{title}{A family
  of skew hadamard difference sets}, \bibinfo{journal}{J. Comb. Theory Ser. A}
  \bibinfo{volume}{113} (\bibinfo{year}{2006}) \bibinfo{pages}{1526--1535}.
\bibitem[{Draper and Hou(2007)}]{DraperHou2007}
\bibinfo{author}{S.~Draper}, \bibinfo{author}{X.~Hou}, \bibinfo{title}{Explicit
  evaluation of certain exponential sums of quadratic functions over
  $\mathbb{F}_{p^n}$, $p$ odd}, \bibinfo{year}{2007}.
  \bibinfo{note}{{arXiv}:0708.3619v1}.
\bibitem[{Edel and Pott(2009)}]{EdelPott2009}
\bibinfo{author}{Y.~Edel}, \bibinfo{author}{A.~Pott}, \bibinfo{title}{A new
  almost perfect nonlinear function which is not quadratic},
  \bibinfo{journal}{Advances in Mathematics of Communications}
  \bibinfo{volume}{3} (\bibinfo{year}{2009}) \bibinfo{pages}{59--81}.
\bibitem[{Ganley(1981)}]{Ganley1981}
\bibinfo{author}{M.~Ganley}, \bibinfo{title}{Central weak nucleus semifields},
  \bibinfo{journal}{European Journal of Combinatorics} \bibinfo{volume}{2}
  (\bibinfo{year}{1981}) \bibinfo{pages}{339--347}.
\bibitem[{Helleseth and Kholosha(2007)}]{Helleseth2007}
\bibinfo{author}{T.~Helleseth}, \bibinfo{author}{A.~Kholosha},
  \bibinfo{title}{On the dual of monomial quadratic p-ary bent functions}, in:
  \bibinfo{editor}{S.~Golomb}, \bibinfo{editor}{G.~Gong},
  \bibinfo{editor}{T.~Helleseth}, \bibinfo{editor}{H.Y. Song} (Eds.),
  \bibinfo{booktitle}{Sequences, Subsequences, and Consequences}, volume
  \bibinfo{volume}{4893} of \textit{\bibinfo{series}{Lecture Notes in Computer
  Science}}, \bibinfo{publisher}{Springer Berlin / Heidelberg},
  \bibinfo{year}{2007}, pp. \bibinfo{pages}{50--61}.
\bibitem[{Hughes and Piper(1973)}]{HughesPiper}
\bibinfo{editor}{D.~Hughes}, \bibinfo{editor}{F.~Piper} (Eds.),
  \bibinfo{title}{Projective Planes}, \bibinfo{publisher}{Springer},
  \bibinfo{address}{Berlin}, \bibinfo{year}{1973}.
\bibitem[{Knuth(1963)}]{Knuth1963}
\bibinfo{author}{D.~Knuth}, \bibinfo{title}{Finite semifields and projective
  planes}, Ph.D. thesis, California Institute of Technology,
  \bibinfo{address}{Pasadena, California}, \bibinfo{year}{1963}.
\bibitem[{Kyureghyan and Pott(2008)}]{EA=CCZ_PN2008}
\bibinfo{author}{G.M. Kyureghyan}, \bibinfo{author}{A.~Pott},
  \bibinfo{title}{Some theorems on planar mappings}, in:
  \bibinfo{booktitle}{WAIFI '08: Proceedings of the 2nd international workshop
  on Arithmetic of Finite Fields}, \bibinfo{publisher}{Springer-Verlag},
  \bibinfo{address}{Berlin, Heidelberg}, \bibinfo{year}{2008}, pp.
  \bibinfo{pages}{117--122}.
\bibitem[{Lidl and Niederreiter(1997)}]{LidlNiederreiter}
\bibinfo{author}{R.~Lidl}, \bibinfo{author}{H.~Niederreiter},
  \bibinfo{title}{Finite fields}, \bibinfo{publisher}{Cambridge University
  Press, Cambridge ; New York :}, \bibinfo{edition}{2nd} edition,
  \bibinfo{year}{1997}.
\bibitem[{Lunardon et~al.(2009)Lunardon, Marino, Polverino and
  Trombetti}]{LMPT2009}
\bibinfo{author}{G.~Lunardon}, \bibinfo{author}{G.~Marino},
  \bibinfo{author}{O.~Polverino}, \bibinfo{author}{R.~Trombetti},
  \bibinfo{title}{Symplectic spreads and quadric {V}eroneseans},
  \bibinfo{year}{2009}. \bibinfo{note}{Manuscript}.
\bibitem[{Penttila and Williams(2004)}]{Penttila-Williams2004}
\bibinfo{author}{T.~Penttila}, \bibinfo{author}{B.~Williams},
  \bibinfo{title}{Ovoids of parabolic spaces}, \bibinfo{journal}{Geometriae
  Dedicata} \bibinfo{volume}{82} (\bibinfo{year}{2004}) \bibinfo{pages}{1--19}.
\bibitem[{Wedderburn(1905)}]{Wedderburn1905}
\bibinfo{author}{J.H.M. Wedderburn}, \bibinfo{title}{A theorem on finite
  algebras}, \bibinfo{journal}{Transaction of the American Mathematical
  Society} \bibinfo{volume}{6} (\bibinfo{year}{1905})
  \bibinfo{pages}{349--352}.
\bibitem[{Weng and Zeng(2010)}]{Weng2010}
\bibinfo{author}{G.~Weng}, \bibinfo{author}{X.~Zeng}, \bibinfo{title}{Further
  results on planar {DO} functions and commutative semifields},
  \bibinfo{year}{2010}. \bibinfo{note}{Submitted}.
\bibitem[{Zha et~al.(2009)Zha, Kyureghyan and Wang}]{Zha2009}
\bibinfo{author}{Z.~Zha}, \bibinfo{author}{G.M. Kyureghyan},
  \bibinfo{author}{X.~Wang}, \bibinfo{title}{Perfect nonlinear binomials and
  their semifields}, \bibinfo{journal}{Finite Fields and Their Applications}
  \bibinfo{volume}{15} (\bibinfo{year}{2009}) \bibinfo{pages}{125 -- 133}.
\bibitem[{Zha and Wang(2009)}]{Zha2009-1}
\bibinfo{author}{Z.~Zha}, \bibinfo{author}{X.~Wang}, \bibinfo{title}{New
  families of perfect nonlinear polynomial functions},
  \bibinfo{journal}{Journal of Algebra} \bibinfo{volume}{322}
  (\bibinfo{year}{2009}) \bibinfo{pages}{3912 -- 3918}.
  \bibinfo{note}{Computational Algebra}.
\bibitem[{Zhou(2010)}]{Zhou2010}
\bibinfo{author}{Y.~Zhou}, \bibinfo{title}{A note on the isotopism of
  commutative semifields}, \bibinfo{year}{2010}.
  \bibinfo{note}{{arXiv}:1006.1529v1}.

\end{thebibliography}







\end{document}